\documentclass[12pt]{amsart}
\usepackage{epsfig,color}

\usepackage{hyperref}
\hypersetup{
    colorlinks=true, %set true if you want colored links
    linktoc=all,     %set to all if you want both sections and subsections linked
    linkcolor=blue,  %choose some color if you want links to stand out
}

\usepackage{mathrsfs}
\usepackage{amssymb}

\theoremstyle{definition}

\newtheorem{theorem}{Theorem}[section]

\newtheorem{proposition}[theorem]{Proposition}
\newtheorem{lemma}[theorem]{Lemma}
\newtheorem{remark}[theorem]{Remark}
\newtheorem{corollary}[theorem]{Corollary}

\numberwithin{equation}{section}

\usepackage{enumerate}

\DeclareMathOperator*{\spa}{span}

\newcommand{\e}{\operatorname{e}}

\newcommand{\pr}{\partial}

\newcommand{\Lap}{\Delta}

\DeclareMathOperator*{\diam}{diam}

\usepackage{comment}

\title[{\L}ojasiewicz inequalities for mean convex self-shrinkers]{{\L}ojasiewicz inequalities for mean convex self-shrinkers}

\author{Jonathan J. Zhu}
\address{Mathematical Sciences Institute, Australian National University, Hanna Neumann Building, Science Road, Canberra, ACT 2601, Australia and Department of Mathematics, Princeton University, Fine Hall, Washington Road, Princeton, NJ 08544, USA}
\email{jjzhu@math.princeton.edu}

\begin{document}
\begin{abstract}
We prove {\L}ojasiewicz inequalities for round cylinders and cylinders over Abresch-Langer curves, using perturbative analysis of a quantity introduced by Colding-Minicozzi. A feature is that this auxiliary quantity allows us to work essentially at first order. This new method interpolates between the higher order perturbative analysis used by the author for certain shrinking cylinders, and the differential geometric method used by Colding-Minicozzi for the round case. 
\end{abstract}
\date{\today}
\maketitle

\section{Introduction}

Self-shrinkers are submanifolds $\Sigma^n\subset\mathbb{R}^N$ satisfying the elliptic PDE $\phi := -\mathbf{H} + \frac{x^\perp}{2}=0$; they serve as singularity models for the mean curvature flow. {\L}ojasiewicz inequalities have been successful for proving the uniqueness of tangent flows for a variety of model shrinkers \cite{S14, CM15, CM19, CS19}, and `explicit' forms can also be used to establish rigidity in the class of shrinkers \cite{ELS18, SZ20}. Explicit {\L}ojasiewicz inequalities for a class of shrinking cylinders were proven by the author in \cite{Z20}, and previously by Colding and Minicozzi for the case of round cylinders \cite{CM15, CM19}. The purpose of this note is to provide a bridge between these two approaches. 

Specifically, in \cite{Z20} we used Taylor expansion of the shrinker quantity $\phi$ while in \cite{CM15,CM19} a pointwise differential geometric method is used, relying on an auxiliary quantity $\tau = \frac{A}{|\mathbf{H}|}$. In this note we show that Taylor expansion of $\tau$, combined with the techniques in \cite{Z20}, yields another {\L}ojasiewicz inequality for round cylinders and Abresch-Langer cylinders. Denoting by $\mathcal{C}_n(\mathring{\Gamma})$ the set of all rotations of $\mathring{\Gamma}^k \times \mathbb{R}^{n-k} \subset \mathbb{R}^N$ about the origin, we prove:

\begin{theorem}[{\L}ojasiewicz inequality of the first kind]
\label{thm:lojasiewicz-tau-intro}
 Let $\mathring{\Gamma}$ be a round shrinking sphere or an Abresch-Langer curve. 
There exists $\epsilon_2>0$ so that for any $\epsilon_1$, $\lambda_0$, $C_j$ there exist $R_0, l_0$  such that if $l\geq l_0$, $\Sigma^n\subset \mathbb{R}^N$ has $\lambda(\Sigma)\leq \lambda_0$ and:
\begin{enumerate}
\item For some $R>R_0$, we have that $B_R\cap \Sigma$ is the graph of a normal field $U$ over some cylinder in $\mathcal{C}_n(\mathring{\Gamma})$ with $\|U\|_{C^{2}(B_R)} \leq \epsilon_2$ and $\|U\|_{L^2(B_R)}^2 \leq \e^{-R^2/8}$; 
\item $|\nabla^j A| \leq C_j$ on $B_R\cap \Sigma$ for all $j \leq l$;
\end{enumerate}
then there is a cylinder $\Gamma \in\mathcal{C}_n(\mathring{\Gamma})$ and a compactly supported normal vector field $V$ over $\Gamma$ with $\|V\|_{C^{2,\alpha}} \leq \epsilon_1$, such that $\Sigma \cap B_{R-6}$ is contained in the graph of $V$, and 

\[\|V\|_{L^2}^2 \leq C(\|\phi\|_{L^1(B_R)}^{a_l} +\|\phi\|_{L^2(B_R)}^{2a_l} + (R-5)^{a_l n} e^{-a_l(R-5)^2/4}),\] %%

where $C=C(n,l,C_l,\lambda_0, \epsilon_1)$ and $a_l \nearrow 1$ as $l\to \infty$. 
\end{theorem}

\begin{theorem}[{\L}ojasiewicz inequality of the second kind]
\label{thm:lojasiewicz-grad} %%move to intro
 Let $\mathring{\Gamma}$ be a round shrinking sphere or an Abresch-Langer curve. 
There exists $\epsilon_2>0$ so that for any $\lambda_0$, $C_j$ there exist $R_0, l_0$ such that if $l\geq l_0$, $\Sigma^n\subset \mathbb{R}^N$ has $\lambda(\Sigma)\leq \lambda_0$ and:
\begin{enumerate}
\item For some $R>R_0$, we have that $B_R\cap \Sigma$ is the graph of a normal field $U$ over some cylinder in $\mathcal{C}_n(\mathring{\Gamma})$ with $\|U\|_{C^{2}(B_R)} \leq \epsilon_2$ and $\|U\|_{L^2(B_R)}^2 \leq \e^{-R^2/8}$; 
\item $|\nabla^j A| \leq C_j$ on $B_R\cap \Sigma$ for all $j \leq l$;
\end{enumerate}
then for $C=C(n,\beta,l,C_l,\lambda_0)$ we have 

\[ |F(\Sigma) - F(\mathring{\Gamma})| \leq C(\|\phi\|_{L^2}^\frac{3a_l}{2} + (R-6)^{n} \e^{-\frac{(R-6)^2}{4}}) .\]
\end{theorem}

Note that by the work of Colding-Minicozzi \cite{CMgeneric}, the only codimension one, mean convex self-shrinkers are precisely the cylinders over round spheres $\mathbb{S}^k_{\sqrt{2k}}$ and Abresch-Langer \cite{AL} curves $\Gamma_{a,b}$. The inequalities above differ slightly from those in \cite{Z20} by the exponents on the right, but morally they are equivalent since $a_l$ may be taken arbitrarily close to 1. Indeed, the above estimates suffice to give alternative proofs of the uniqueness of tangent flows and rigidity for these mean convex shrinkers (see Remark \ref{rmk:applications}; cf. \cite[Theorems 1.1 and 1.2]{Z20}). 

The advantage of the $\tau$ quantity is that it allows us to perform the variational analysis only at first order, whereas in \cite{Z20} we needed the second order expansion of $\phi$. This is a significant reduction as the complexity of the method increases quickly with the order of expansion. Applying the perturbative method to $\tau$ also explains and quantifies the success of the method used by Colding-Minicozzi \cite{CM15, CM19} for round spheres. 

The key new geometric data, derived in Section \ref{sec:var}, are the (first) variation formulae for $\tau$ and a further auxiliary quantity $P$. We then prove estimates for entire graphs over the cylinders in Section \ref{sec:entire}, using Taylor expansion of $\tau$ in place of the second order analysis in \cite[Section 5.2]{Z20}. Note that the first order analysis of $\phi$ is still required. Some preliminaries are included in Section \ref{sec:prelim} and the {\L}ojasiewicz inequalities of Theorems \ref{thm:lojasiewicz-tau-intro} and \ref{thm:lojasiewicz-grad} are proven in Section \ref{sec:main-est}. 

\subsection*{Acknowledgements}

The author would like to thank Prof. Bill Minicozzi for his encouragement and for several insightful discussions. This work was supported in part by the National Science Foundation under grant DMS-1802984 and the Australian Research Council under grant FL150100126. 

%\tableofcontents

\section{Preliminaries}
\label{sec:prelim}

We consider smooth, properly immersed submanifolds $\Sigma^n\subset \mathbb{R}^N$. 
%We use $x$ to denote the position vector on $\mathbb{R}^N$, although sometimes it will be useful to denote a given immersion by $X: \Sigma\rightarrow \mathbb{R}^N$, particularly in Section \ref{sec:variation}. For instance, while $x(p)=X(p)$, we use $x_i$ to denote the coordinate functions whereas $X_i$ may refer to derivatives of the immersion $X$. 

For a vector $V$ we denote by $V^T$ the projection to the tangent bundle, and $V^\perp = \Pi(V)$ the projection to the normal bundle $N\Sigma$. Given a vector field $U$ on $\Sigma$ with $\|U\|_{C^1}$ small enough, the graph $\Sigma_U$ is the submanifold given by the immersion $X_U(p)= X(p) + U(p)$. We say $\Sigma_U$ is a normal graph if $U^T=0$. 

The second fundamental form is the 2-tensor with values in the normal bundle defined by $A(Y,Z) = \nabla^\perp_Y Z$, and the mean curvature (vector) is $\mathbf{H} =  - A_{ii} $. 
Here, and henceforth, we take the convention that repeated lower indices are summed with the metric, for instance $A_{ii} = g^{ij} A_{ij}$. 
We denote the shrinker mean curvature by $\phi = \frac{1}{2}x^\perp-\mathbf{H}$ and the principal normal by $\mathbf{N} = \frac{\mathbf{H}}{|\mathbf{H}|}$. A submanifold is a shrinker if $\phi\equiv0$ on $\Sigma$. 

Given a vector $V$ we denote $A^V = \langle A,V\rangle$. The Hessian on the normal bundle is given by $(\nabla^\perp\nabla^\perp V)(Y,Z) = \nabla^\perp_Z \nabla^\perp_Y V - \nabla^\perp_{\nabla^T_Z Y} V$. 

For graphs $\Gamma_U$ over a fixed submanifold $\Gamma$, we use subscripts to denote the values of geometric quantities on $\Gamma_U$. We also consider these quantities as second order functionals on (normal) vector fields $U$. For instance, there is a smooth function $\mathcal{\varphi}$ such that $\phi_U =\mathcal{\varphi}(p,U,\nabla U,\nabla^2 U)$. For variations of such quantities, we use the shorthand notation $\mathcal{D}\mathcal{\varphi}(U)$ to mean the variation $\mathcal{D} \mathcal{\varphi}|_0 ([U,\nabla U,\nabla^2 U])$ evaluated at 0, and so forth. 

The Gaussian weight is $\rho=\rho_n=(4\pi)^{-n/2} \e^{-|x|^2/4}$. Here $n$ is the dimension of the submanifold and will be omitted when clear from context. By $L^p$, $W^{k,p}$ we denote the weighted Sobolev spaces with respect to $\rho$. The Gaussian area functional is $F(\Sigma) = \int_\Sigma \rho$. The entropy is $\lambda(\Sigma) = \sup_{y,s>0} F(s(\Sigma-y))$. For a shrinker, $\lambda(\Sigma)=F(\Sigma)$. Note that finite entropy $\lambda(\Sigma)\leq \lambda_0$ implies Euclidean volume growth $|\Sigma \cap B_R| \leq C(\lambda_0) R^n$. 

We will use the following elliptic operators: the drift Laplacian $\mathcal{L} = \Lap -\frac{1}{2}\nabla_{x^T}$; and the Jacobi operator $L = \mathcal{L} +\frac{1}{2} + \sum_{k,l} \langle \cdot,A_{kl}\rangle A_{kl}$. The drift Laplacian is defined on functions and tensors, whilst $L$ is defined on sections of the normal bundle (via $\nabla^\perp$). For such operators, unless otherwise indicated, $\ker$ will refer to the $W^{2,2}$ kernel, for instance $\mathcal{K}= \ker L$. 

We set $\langle x\rangle = (1+|x|^2)^\frac{1}{2}$. On a curve $\Gamma^1\subset\mathbb{R}^2$, we denote the geodesic curvature by $\kappa$ and use dots $\dot{\kappa}=\pr_\sigma \kappa$ to denote differentiation with respect to the arclength parameter $\sigma$. 

We use $C$ to denote a constant that may change from line to line but retains the stated dependencies. 

\subsection{Mean convex self-shrinkers}

In this article, we say that a submanifold $\Gamma^n\subset \mathbb{R}^N$ has `codimension one' if the minimal affine subspace containing $\Gamma$ has dimension $\dim \spa(\Gamma) = n+1$. Note that for shrinkers, the minimal subspace necessarily contains the origin since $\mathcal{L}x = -\frac{1}{2}x$. Consider a codimension one shrinker $\Gamma$. Up to ambient rotation we have $\Gamma^n \subset \mathbb{R}^{n+1} \times \mathbb{R}^{N-n-1}$. Moreover the normal bundle is trivial and is spanned by $\mathbf{N}$ and $\pr_{z_\alpha}$, where $z_\alpha$ are standard coordinates on $\mathbb{R}^{N-n-1}$. 

An orientable codimension one submanifold is mean convex (up to change of orientation) if $|\mathbf{H}| >0$. By the work of Colding-Minicozzi \cite{CMgeneric} (see also \cite{Hui90}), the only mean convex self-shrinkers with finite entropy are cylinders $\Gamma=\mathring{\Gamma}^k\times \mathbb{R}^{n-k}$, where $\mathring{\Gamma}$ is either a round shrinking sphere $\mathbb{S}^k_{\sqrt{2k}}$ or an Abresch-Langer curve $\mathring{\Gamma}^1_{a,b}$ (see \cite{AL}). We further decompose $\mathbb{R}^N = \mathbb{R}^{k+1} \times \mathbb{R}^{n-k} \times \mathbb{R}^{N-m-n}$ so that $\mathring{\Gamma}\subset \mathbb{R}^{k+1}$, and let $\mathring{x}, y,z$ be the projection of $x$ to each respective factor. 

Given $\mathring{\Gamma}^k$, we denote by $\mathcal{C}_n(\mathring{\Gamma})$ the set of all rotations of $\mathring{\Gamma} \times \mathbb{R}^{n-k} \subset \mathbb{R}^N$ about the origin.

\subsection{The auxiliary quantities $\tau$ and $P$}

For submanifolds $\Gamma$ on which $\mathbf{H}$ never vanishes, Colding-Minicozzi \cite{CM19} considered the 2-tensor $\tau = \frac{A}{|\mathbf{H}|}$ and showed that $|\nabla^\perp\tau|^2$ satisfies a certain elliptic PDE with inhomogenous term given by
\begin{equation}
\begin{split}
P ={} & |A|^2 |A^\mathbf{N}|^2 - 2|A^2|^2 + \sum_{ijlm} \left(2\langle A_{jl}, A_{im}\rangle \langle A_{lm}, A_{ij}\rangle - \langle A_{ij},A_{ml}\rangle^2\right) \\& +\frac{|A|^2}{4|\mathbf{H}|^2}\left(|A^\mathbf{N}(x^T,\cdot)|^2 - |A(x^T,\cdot)|^2\right).
\end{split}
\end{equation}

Here $A^2$ is the real-valued 2-tensor $(A^2)_{ij} = \langle A_{im}, A_{mj}\rangle$. 

If $\Gamma = \mathring{\Gamma}^k\times \mathbb{R}^{n-k}$ is either a round cylinder or an Abresch-Langer cylinder, one has $\tau = -\frac{1}{k} \mathbf{N} \mathring{g}_{ij}$ and in particular $\nabla^\perp\tau=0$. 

As before, for graphs $\Gamma_U$ there are smooth functions $\mathcal{T}$ and $\mathcal{P}$ so that for sufficiently small $\delta>0$ and $\|U\|_{C^2}<\delta$ we have $|\nabla^\perp \tau|^2_U = \mathcal{T}(p,U,\nabla U,\nabla^2 U)$ and $P_U = \mathcal{P}(p,U,\nabla U,\nabla^2 U)$. 

The quantity $P$ vanishes on submanifolds of codimension one. 

\subsection{Jacobi fields} 

The space of Jacobi fields on a shrinker $\Gamma$ is the ($W^{2,2}$) kernel $\mathcal{K}=\ker L$. It contains the subspace $\mathcal{K}_0$ of Jacobi fields generated by ambient rotations, and we denote its $L^2$-orthocomplement in $\mathcal{K}$ by $\mathcal{K}_1$. 

The following summarises the Jacobi fields on $\Gamma = \mathring{\Gamma}\times \mathbb{R}^{n-k}$, where $\mathring{\Gamma}$ is a round shrinking sphere or an Abresch-Langer curve (see \cite[Sections 2.6 and 4.1]{Z20}). 

\begin{proposition}
\label{prop:jacobi-fields}
Let $\Gamma =\mathring{\Gamma} \times \mathbb{R}^{n-k}$ where $\mathring{\Gamma}$ is a round shrinking sphere or an Abresch-Langer curve. Then the space $\mathcal{K}_0$ is spanned by normal vector fields of the following forms: 
\begin{enumerate}
\item $\mathring{x}_i\pr_{\mathring{x}_j}^\perp - \mathring{x}_j\pr_{\mathring{x}_i}^\perp$; 
\item $y_j \pr_{\mathring{x}_i}^\perp$; 
\item $\mathring{x}_i \pr_{z_\alpha}$; and $y_j \pr_{z_\alpha}$. 
\end{enumerate}
Moreover, the space $\mathcal{K}_1$ is spanned by the normal fields $\{(y_i y_j-2\delta_{ij})\mathbf{H}\}$. 
\end{proposition}

\begin{corollary}
\label{cor:jacobi-est}
Let $\Gamma =\mathring{\Gamma} \times \mathbb{R}^{n-k}$ where $\mathring{\Gamma}$ is a round shrinking sphere or an Abresch-Langer curve. Let $r_0 = \diam(\mathring{\Gamma})+1$. There exists $C$ so that for any $J \in \mathcal{K}$ we have $|J| \leq C\langle x\rangle^2 \|J\|_{L^2(B_{r_0})}$, $|\nabla J| + |\nabla^2 J| \leq C\langle x\rangle^2 \|J\|_{L^2(B_{r_0})}$ and $|\nabla^2 J ( \cdot,y)\|  \leq C\langle x\rangle\|J\|_{L^2(B_{r_0})} $. 
\end{corollary}
 
 \section{Auxiliary variation analysis}
 \label{sec:var}
 
In this section we compute the variation of the auxiliary quantities, insofar as to establish Propositions \ref{prop:dDtauS} and \ref{prop:dDtauAL} for $\nabla^\perp\tau$, and Lemma \ref{lem:dPAL} and Proposition \ref{prop:d2PAL} for $P$. The benefit of using these quantities is that their first variation will be sufficient to establish the formal second order obstruction of $\phi$. 

We consider a submanifold $\Gamma$ with a fixed immersion $X_0 :\Gamma^n \rightarrow \mathbb{R}^N$, and a one-parameter family of immersions $X: I\times \Gamma^n\rightarrow \mathbb{R}^N$ with $X(0,p)=X_0(p)$. We use $s$ for the coordinate on $I=(-\epsilon,\epsilon)$, and subscripts to denote differentiation with respect to $s$. For instance, $X_s = \frac{\pr}{\pr s} X$. If $p_i$ are local coordinates on $\Sigma$, we have the tangent frame $X_i = X_*(\frac{\pr}{\pr p^i})$. 

All geometric quantities such as $\Pi, g,A$ should be considered as functions of $s,p$, given by the value of each quantity at $X(s,p)$ on the submanifold defined by $X(s,\cdot)$. For instance, the metric $g_{ij}(s,p)$ is given by $g_{ij}=\langle X_i, X_j\rangle$. Recall $\Pi$ is the projection to the normal bundle. Also recall that repeated lower indices are contracted via the (inverse) metric $g^{ij}$.

The following first variations were calculated in \cite{CM19}:

\begin{proposition}[\cite{CM19}]
\label{prop:1stvar}
Suppose $X_s=V=V^\perp$; then at $s=0$:

\begin{alignat}{3}
\label{eq:dPi}
&\Pi_s(W) &&=  -\Pi(\nabla_{W^T} V)  - X_j g^{ij}\langle \Pi(\nabla_{X_i} V), W\rangle,\\
\label{eq:dg}
&(g_{ij})_s &&=  -2A^V_{ij}, (g^{ij})_s = 2g^{il} A^V_{lm} g^{mj}, \\
\label{eq:dA}
&(A_{ij})_s &&= -X_l \langle \nabla^\perp_{X_l} V, A_{ij}\rangle + (\nabla^\perp\nabla^\perp V)(X_i,X_j) - A^V_{il} A_{jl},%\\
% \label{eq:dphi}
% & \phi_s = \mathcal{D}\mathcal{\varphi} (V)&&= LV - X_j g^{ij} \langle \nabla^\perp_{X_i} V,\phi\rangle.
\\
&|\mathbf{H}|_s &&= -\langle \mathbf{N}, \Lap^\perp V + A^V_{ij}A_{ij}\rangle.
  \end{alignat}
\end{proposition}

\subsection{Variation of $\nabla^\perp\tau$}
 
We begin with a general submanifold $\Gamma$. Assume that on $\Gamma$, we have $|\mathbf{H}|\neq 0$ on $\Gamma$ (so $\tau$ is well-defined) and $\nabla^\perp\tau=0$. Then at $s=0$, \begin{equation}(|\nabla^\perp\tau|^2)_s =0, \qquad (|\nabla^\perp \tau|^2)_{ss} = 2|(\nabla^\perp\tau)_s|^2.\end{equation} The first variation of $\nabla^\perp \tau$ is given by \begin{equation}\label{eq:dDtau}(\nabla^\perp \tau)_s = \Pi_s (\nabla^\perp\tau) + \Pi(\nabla(\tau_s)).\end{equation} 

%Now assume that $A=A^\mathbf{N}\mathbf{N}$ and hence $\tau = \tau^{\mathbf{N}} \mathbf{N}$ on $\Gamma$. 
Since $\nabla^\perp\tau=0$, we have $\nabla_l \tau_{ij} = - \langle \tau_{ij}, A_{lm}\rangle e_m.$ Using (\ref{eq:dPi}), this gives the following formula for the first term on the right in (\ref{eq:dDtau}),  \begin{equation}\label{eq:dPiDtau}\Pi_s(\nabla_l \tau_{ij}) = \frac{1}{|\mathbf{H}|} \langle A_{ij}, A_{lm}\rangle \nabla_m^\perp V.\end{equation} 

For the last term in (\ref{eq:dDtau}), we calculate \begin{equation}\label{eq:Ddtau}\Pi(\nabla (\tau_s)) = - \frac{\nabla |\mathbf{H}|}{|\mathbf{H}|^2} \Pi(A_s - |\mathbf{H}|_s \tau) + \frac{\nabla^\perp A_s}{|\mathbf{H}|} - \frac{\nabla |\mathbf{H}|_s}{|\mathbf{H}|} \tau.\end{equation} 

The first two terms in (\ref{eq:Ddtau}) are already known from Proposition \ref{prop:1stvar}, and we proceed to calculate the last two: Differentiating $|\mathbf{H}|_s = -\langle \mathbf{N}, \Lap^\perp V + A^V_{ij}A_{ij}\rangle$ gives \begin{equation}
\begin{split}\nabla |\mathbf{H}|_s =& -\langle \nabla^\perp \mathbf{N}, \Lap^\perp V + A^V_{ij}A_{ij}\rangle - \langle \mathbf{N}, \nabla^\perp \Lap^\perp V\rangle \\&-A^{\mathbf{N}}_{ij}( \langle\nabla^\perp A_{ij},V\rangle +\langle A_{ij},\nabla^\perp V\rangle) - A^V_{ij}\langle \mathbf{N}, \nabla^\perp A_{ij}\rangle.\end{split}\end{equation}

Now differentiating $(A_{ij})_s = -X_l \langle \nabla_l^\perp V,A_{ij}\rangle - A^V_{il} A_{jl} + (\nabla^\perp\nabla^\perp V)(X_i,X_j) $, we have \begin{equation}\begin{split}\nabla_l^\perp (A_{ij})_s =& -A_{ml} \langle \nabla_m^\perp V,A_{ij}\rangle - \langle \nabla_l^\perp A_{im},V\rangle A_{jm} - \langle A_{im},\nabla_l^\perp V\rangle A_{jm} \\& - A^V_{im}\nabla_l^\perp A_{jm} + (\nabla^\perp \nabla^\perp \nabla^\perp V)(X_i,X_j,X_l).\end{split}\end{equation}

 In the remainder of this subsection we consider the cases where $\mathring{\Gamma}$ is either: a shrinking sphere $\mathbb{S}^k_{\sqrt{2k}}$; or an Abresch-Langer curve $\mathring{\Gamma}_{a,b}$. In both cases, $\Gamma$ has codimension one and indeed satisfies $\nabla^\perp \tau=0$ and $A=A^\mathbf{N} \mathbf{N}$. Furthermore, $\nabla^\perp\mathbf{N}=0$, and it follows that any normal variation on $\Gamma$ may be written $U= u\mathbf{N} + u^\alpha \pr_{z_\alpha}$, with $\nabla^\perp_i U = (\nabla_i u) \mathbf{N} + (\nabla_i u^\alpha) \pr_{z_\alpha}$, and so forth. 
  
  \subsubsection{Round cylinders}
  
    \begin{proposition}
  \label{prop:dDtauS}
  Let $\mathring{\Gamma} = \mathbb{S}^k_{\sqrt{2k}}$ and $\Gamma \in \mathcal{C}_n(\mathring{\Gamma})$. If $U\in \mathcal{K}_1$, then $\|\mathcal{D}^2\mathcal{T}(U,U)\|_{L^1} = \frac{2}{k^3} \|U\|_{L^2}^2 $.
  \end{proposition}
  \begin{proof}
  
A shrinking sphere $\mathring{\Gamma} = \mathbb{S}^k_{\sqrt{2k}}$ satisfies $\mathring{A}_{ij} = -\frac{1}{\sqrt{2k}} \mathbf{N} \mathring{g}_{ij}$. Consider a variation by $U=u\mathbf{N}$. As in \cite{CM19}, it follows that at $s=0$ we have 

\[(A_{ij})_s= \frac{\mathring{g}_{ij}}{\sqrt{2k}}\nabla u - \frac{\mathring{g}_{ij}}{2k} u \mathbf{N} + (\nabla^2 u)_{ij}\mathbf{N} ,\]
\[|\mathbf{H}|_s=-\Lap u - \frac{u}{2}.\]

Using the variation formulae above, we may also compute that:

\[\Pi_s(\nabla_l\tau_{ij}) =  \frac{1}{\sqrt{2k^3}} \mathring{g}_{ij}\mathring{g}_{ml} (\nabla_m u) \mathbf{N},\]

\[\nabla_l^\perp (A_{ij})_s = -\frac{1}{2k} \mathring{g}_{ij} \mathring{g}_{ml} (\nabla_m u) \mathbf{N}  -\frac{1}{2k} \mathring{g}_{ij} (\nabla_l u) \mathbf{N} + (\nabla^3 u)_{ijl} \mathbf{N},\] 

\[\nabla_l |\mathbf{H}|_s = -\nabla_l \Lap u -\frac{1}{2}\nabla_l u.\] 

We now specialise to $U\in \mathcal{K}_1$, so that $u=u(y) = \sum_{ij} c_{ij} (y_iy_j-2\delta_{ij})$. In particular $(\mathcal{L}+1)u=0$ and $\nabla^3 u =0$. 

Combining the above according to (\ref{eq:dDtau}), (\ref{eq:dPiDtau}) and (\ref{eq:Ddtau}) then gives \begin{equation}(\nabla^\perp\tau)_{ijl,s} = \sqrt{\frac{2}{k^3}} \mathring{g}_{ij} (\nabla_l u - \mathring{g}_{lm}\nabla_m u) \mathbf{N}\end{equation} and therefore  \begin{equation}\frac{1}{2}(|\nabla^\perp\tau|^2)_{ss} = \frac{1}{2}\mathcal{D}^2\mathcal{T}(U,U) = \frac{2}{k^3} |\nabla u|^2.\end{equation}

Using the identity $\int_\Gamma |\nabla u|^2 \rho = -\int_\Gamma u(\mathcal{L}u)\rho = - \int_\Gamma u^2\rho$ then completes the proof. 

\end{proof}
 
 \subsubsection{Abresch-Langer curves}
 
   \begin{proposition}
  \label{prop:dDtauAL}
  Let $\mathring{\Gamma} = \mathring{\Gamma}_{a,b}$ be an Abresch-Langer curve and $\Gamma \in \mathcal{C}_n(\mathring{\Gamma})$. 
  
  If $U\in \mathcal{K}_1$, then  $\|\mathcal{D}^2\mathcal{T}(U,U)\|_{L^1} \geq 4 B_2(\mathring{\Gamma}) \|U\|_{L^2}^2 $, where $B_2(\mathring{\Gamma}) = \frac{\int_{\mathring{\Gamma}} \kappa^4 \rho_2}{\int_{\mathring{\Gamma}} \kappa^2 \rho_2}>0$. 
  \end{proposition}
  \begin{proof}
 
The curve $\mathring{\Gamma} = \mathring{\Gamma}_{a,b}$ satisfies $\mathring{A}_{ij} = -\kappa \mathbf{N} \mathring{g}_{ij}$, where $\kappa$ is the geodesic curvature. Consider a variation by $U=u\mathbf{N}$. It follows that at $s=0$ we have:

\[\Pi_s(\nabla_l\tau_{ij}) = \kappa \mathring{g}_{ij}\mathring{g}_{ml} (\nabla_m u) \mathbf{N},\]

\[(A_{ij})_s = \kappa \mathring{g}_{ij} \nabla u - \kappa^2 u \mathring{g}_{ij} \mathbf{N} + (\nabla^2u)_{ij} \mathbf{N},\]

\[\nabla_l^\perp (A_{ij})_s = -\kappa^2 \mathring{g}_{ij} \mathring{g}_{ml} (\nabla_m u) \mathbf{N} - 2\kappa \mathring{g}_{ij} u (\nabla_l \kappa) \mathbf{N} - \kappa^2 \mathring{g}_{ij} \nabla_l u \mathbf{N} + (\nabla^3 u)_{ijl} \mathbf{N},\] 

\[|\mathbf{H}|_s =-\Lap u - \kappa^2 u, \qquad \nabla_l |\mathbf{H}|_s = -\nabla_l \Lap u - 2\kappa u\nabla_l \kappa - \kappa^2\nabla_l u.\] 

We now specialise to $U\in \mathcal{K}_1$, so that $u= f(y)\kappa$ and $f(y) = \sum_{ij} c_{ij} (y_iy_j-2\delta_{ij})$; in particular note that $(\mathcal{L}+1)f=0$. 

Take coordinates on $\Gamma$ so that the $i=0$ index corresponds to the arclength parameter on $\mathring{\Gamma}$ and the remaining indices $i>0$ correspond to the standard coordinates on $\mathbb{R}^{n-1}$. Then the metric on $\mathring{\Gamma}$ satisfies $\mathring{g}_{ij} = \delta_{ij}$. 

Combining the above according to (\ref{eq:dDtau}), (\ref{eq:dPiDtau}) and (\ref{eq:Ddtau}) then gives, for $i=j=0$ and $l>0$, that \begin{align*}(\nabla^\perp\tau)_{00l,s} &= (-2\kappa \nabla_l u - \kappa^{-1}\nabla_l \Lap u + \kappa^{-1} (\nabla^3 u)_{00l})\mathbf{N} \\&= (-2\kappa^2 \nabla_l f - \kappa^{-1}\nabla_l (\ddot{\kappa}f + 2\kappa)+\kappa^{-1} \ddot{\kappa} \nabla_l f)\mathbf{N} \\&= -2\kappa^2 (\nabla_l f) \mathbf{N}.\end{align*} Therefore \begin{equation}\frac{1}{2}(|\nabla^\perp\tau|^2)_{ss} = \frac{1}{2}\mathcal{D}^2\mathcal{T}(U,U) = |(\nabla^\perp\tau)_s|^2 \geq 4\kappa^4|\nabla f|^2.\end{equation}

Again using the identity $\int_\Gamma |\nabla f|^2 \rho = -\int_\Gamma f(\mathcal{L}f)\rho = - \int_\Gamma f^2\rho$ completes the proof. 

\end{proof}

\subsection{Variation of $P$}

The first and second variations of $P$ were studied in \cite{CM19} for the case of round cylinders. Here we consider the case where $\mathring{\Gamma}$ is an Abresch-Langer curve. Again we write normal variations as $U= u\mathbf{N} + u^\alpha \pr_{z_\alpha}$. 

\begin{lemma} 
\label{lem:dPAL}
$\mathcal{D} \mathcal{P}=0$ on $\Gamma = \mathring{\Gamma}_{a,b}\times \mathbb{R}^{n-1}$. 
\end{lemma}
\begin{proof}
The variation of the first four terms of $P$ proceeds similarly to the round cylinder case in \cite[Section 5.3]{CM19}, so we only list the results of some key calculations. The basic ingredients are (evaluated at $s=0$):

\begin{equation}\langle A_{ij}, A_{ml} \rangle = \kappa^2 \mathring{g}_{ij}\mathring{g}_{ml},  \qquad
\langle (A_{ij})_s, A_{ml}\rangle = -\kappa(u_{ij}-\kappa^2 \mathring{g}_{ij} u) \mathring{g}_{ml},\end{equation}
\begin{equation}(g^{ij})_s = 2A^V_{ij} -2\kappa u\mathring{g}^{ij},\end{equation}
\begin{equation}-\mathbf{H}_s = \kappa\nabla u + (\Lap u + \kappa^2 u)\mathbf{N} + (\Lap u^\alpha) \pr_{z_\alpha}, \qquad
\mathbf{N}_s = -\nabla u-\frac{1}{\kappa}(\Lap u^\alpha)\pr_{z_\alpha}.\end{equation}

Combining these as in \cite{CM19} gives

\[(|A|^2)_s = -2\kappa(\ddot{u} + \kappa^2u),\]
\[\langle A_{ij}, \mathbf{N}\rangle_s = u_{ij} - \kappa^2 \mathring{g}_{ij} u,\qquad
(|A^\mathbf{N}|^2)_s = -2\kappa(\ddot{u} + \kappa^2u),\] 
\[(A^2_{ij})_s = -\kappa(u_{im}\mathring{g}_{mj} + u_{mj} \mathring{g}_{im}), \qquad 
(|A^2|^2)_s = -4\kappa^3(\ddot{u}+\kappa^2u),\] 
\[(\langle A_{ij}, A_{ml}\rangle^2)_s = -4\kappa^3(\ddot{u}+\kappa^2 u),\]
 \[(\langle A_{jl},A_{im}\rangle\langle A_{lm},A_{ij}\rangle)_s = -4\kappa^3(\ddot{u}+\kappa^2 u).\] 
 
Since $|A|^2 = |A^\mathbf{N}|^2 = \kappa^2$, from the above one can see that the first variation of the first four terms of $P$ cancels to zero. It remains to check the variation of the last term \[\frac{|A|^2}{4|\mathbf{H}|^2}(|A^\mathbf{N}(x^T,\cdot)|^2) - |A(x^T,\cdot)|^2).\] Since $A = A^\mathbf{N} \mathbf{N}$, the factor in brackets vanishes at $s=0$, so it is enough to show that its variation is zero too. 

Indeed, we find that $(A^\mathbf{N}(x^T,\cdot))_s = \langle (A(x^T,\cdot))_s, \mathbf{N}\rangle + \langle A(x^T,\cdot), \mathbf{N}_s\rangle = \langle (A(x^T,\cdot))_s, \mathbf{N}\rangle$ since $\langle \mathbf{N}_s ,\mathbf{N}\rangle=0$. Therefore we have \[(|A^\mathbf{N}(x^T,\cdot)|^2)_s = 2 \langle (A^\mathbf{N}(x^T,\cdot))_s , A^\mathbf{N}(x^T,\cdot)\rangle = 2 \langle (A(x^T,\cdot))_s,  A(x^T,\cdot)\rangle =(|A(x^T,\cdot)|^2)_s.\] This completes the proof. 
\end{proof}

Using Lemma \ref{lem:dPAL}, the same proof as \cite[Corollary 5.36]{CM19} (using that $\mathring{\Gamma}$ and its variations by $\mathcal{K}_1$ have codimension one) now gives 
\begin{proposition}
\label{prop:d2PAL}
Let $\Gamma = \mathring{\Gamma}_{a,b}\times \mathbb{R}^{n-1}$. 
If $V\in N\Gamma$ with $LV=0$ and $\|V\|_{W^{2,2}}<\infty$ then $(\mathcal{D}^2 \mathcal{P})(V,V)=0$. 
\end{proposition}

\begin{remark}
One may also check directly that $\mathcal{D}^2\mathcal{P}(V,V)=0$ for $V\in \mathcal{K}$ using the second variation formulae in \cite[Section 3]{Z20}.
\end{remark}

\section{Estimates for entire graphs}
\label{sec:entire}

We now proceed to prove estimates for entire graphs using the auxiliary quantities. The setup for this subsection follows that of Section \ref{sec:entire}. In particular we consider a normal variation fields $U$ with compact support over a cylinder $\Gamma = \mathring{\Gamma}\times \mathbb{R}^{n-k}$ with $\|U\|_{C^2}<\delta$. 

Set $|V|_{m} =\sum_{j\leq m} |\nabla^j V|$, so that $\|V\|_{W^{m,q}}^q = \int_\Gamma |V|_m^q \rho$. 

First, we note the following crude bounds to be used for Taylor expansion.

\begin{lemma}\label{lem:C3tauP}
Let $\mathring{\Gamma}$ be a compact shrinker with $|\mathring{\mathbf{H}}|>0$. If $\Gamma = \mathring{\Gamma}\times \mathbb{R}^{n-k}$, then there exists $C,\delta$ such that for any vector field $U$ on $\Gamma$ with $\|U\|_{C^2} <\delta$, we have $\|\mathcal{T}\|_{C^3} \leq C$ and $\|\mathcal{P}\|_{C^3} \leq C\langle x\rangle^2$. 
  \end{lemma}
\begin{proof}
As in \cite[Lemma 5.30]{CM19}, each quantity in the definition of $\nabla^\perp\tau$ and $P$ is a smooth function of $(p,U,\nabla U, \nabla^2 U)$. The position vector $x(p)$ does not enter the definition of $\nabla^\perp \tau$ but it enters the definition of $P$ quadratically. 
\end{proof}

In the remainder of this section, $\mathring{\Gamma}$ is either a round sphere or an Abresch-Langer curve. 

\subsection{First order decomposition}
\label{sec:decomp}

Given a compactly supported normal field $U$ on $\Gamma$, we have the orthogonal decomposition $U= J+h$, where $J \in \mathcal{K}$ and $h\in \mathcal{K}^\perp$. We may further decompose $J = U_0 + J'$, where $U_0 \in \mathcal{K}_0$ and $J' \in \mathcal{K}_1$. 

The first order expansion of $\phi$ implies the following estimates (see \cite[Proposition 5.4]{Z20}):

\begin{proposition}[\cite{Z20}]
\label{prop:ord1}
Let $\Gamma=\mathring{\Gamma}\times \mathbb{R}^{n-k}$. There exists $\epsilon_0$ and $C$ so that if $U$ is a compactly supported normal field on $\Gamma$ with $\|U\|_{C^2}<\epsilon_0$, then
\begin{equation}
\|h\|_{W^{2,2}}^2 \leq C(\|\phi_U\|_{L^2}^2 + \|U\|_{L^2}^4), 
\end{equation}
and for any $\kappa\in(0,1]$, 
\begin{equation}
\int_\Gamma \langle x\rangle^6 |U|_2^3\rho \leq C(\kappa)(\|\phi_U\|_{L^2}^\frac{6}{3+\kappa} + \|U\|_{L^2}^3).
\end{equation}
\end{proposition}

\subsection{Estimates by $\nabla^\perp\tau$ and $\phi$}

We proceed with $U = J+h = U_0 + J'+h$ decomposed as in Section \ref{sec:decomp}. 

\begin{lemma}
There exists $C$ so that for any $U$ as above, we have the pointwise estimate
\begin{equation}
\label{eq:taylor-est-tau}
| |\nabla^\perp\tau|^2_U - \frac{1}{2}\mathcal{D}^2\mathcal{T}(J', J') | \leq C( |U|_2^3+ 2|J|_2|h|_2 +|h|_2^2 + 2|J'|_2|U_0|_2  + |U_0|_2^2). 
\end{equation}
\end{lemma}
\begin{proof}
Let $T(s) = \mathcal{T}(p,sU,s\nabla U, s\nabla^2 U)$. Since $\|\mathcal{T}\|_{C^3} \leq C$ by Lemma \ref{lem:C3tauP}, Taylor expansion about $s=0$ gives \[ |T(1) -T(0) - T'(0) - \frac{1}{2}T''(0)| \leq C |U|_2^3.\] 

Note that $T(1) = |\nabla^\perp\tau|^2_U$, $T(0)=0$, $T'(0) = \mathcal{DT}(U)=0$ and $T''(0) = \frac{1}{2}\mathcal{D}^2\mathcal{T}(U,U)$. Expanding the bilinear form $\mathcal{D}^2\mathcal{T}(U,U)$ according to the decomposition of $U$, and using $\|\mathcal{T}\|_{C^3}\leq C$ to estimate the remaining terms except $\mathcal{D}^2\mathcal{T}(J',J')$ finishes the proof. 
\end{proof}

We may now estimate the variation field $U$ in terms of $\phi_U$ and $|\nabla^\perp\tau|_U$. 

\begin{proposition}
\label{prop:entire-est-tau-phi}%fix statement
Let $\Gamma = \mathring{\Gamma}\times \mathbb{R}^{n-k}$, where $\mathring{\Gamma}$ is a round shrinking sphere or an Abresch-Langer curve. There exists $\epsilon_0>0$ such that if $\kappa \in (0,1]$ and $U$ is a compactly supported normal vector field on $\Gamma$ with $\|U\|_{C^2}\leq \epsilon_0$, then 
  \begin{equation}
 \|U\|_{L^2}^2 \leq C(\kappa) ( \|U_0\|_{L^2}^2+ \| |\nabla^\perp\tau|^2_U\|_{L^1}   + \|\phi_U\|_{L^2}^\frac{6}{3+\kappa}),
 \end{equation}
 where $U_0 = \pi_{\mathcal{K}_0}(U)$. 
 \end{proposition}
 \begin{proof}

By Proposition \ref{prop:dDtauS} or \ref{prop:dDtauAL} respectively we have $\delta_0 \|J'\|_{L^2}^2 \leq \|\mathcal{D}^2 \mathcal{T}(J',J')\|_{L^1}$ for some $\delta_0>0$. Now to estimate the right hand side, we integrate estimate (\ref{eq:taylor-est-tau}), using Corollary \ref{cor:jacobi-est} to estimate the Jacobi terms. This gives

\begin{equation}
\begin{split}
C^{-1}\|\mathcal{D}^2 \mathcal{T}(J',J')\|_{L^1}\leq{}& \| |\nabla^\perp\tau|^2_U\|_{L^1} + \|U\|_{W^{2,3}}^3+ \|h\|_{W^{2,2}}^2+\|U_0\|_{L^2}^2 \\&+\|U\|_{L^2} \int_\Gamma \langle x\rangle^2 |h|_2 \rho + \|U\|_{L^2}\|U_0\|_{L^2}\\
\leq{}& \| |\nabla^\perp\tau|^2_U\|_{L^1} + \|U\|_{W^{2,3}}^3+ (1+\epsilon^{-1}) \|h\|_{W^{2,2}}^2+\|U_0\|_{L^2}^2 \\&+2\epsilon \|U\|_{L^2} + \epsilon^{-1} \|U_0\|_{L^2}^2. 
\end{split}
\end{equation}

Here for the second line we have used Cauchy-Schwarz, so that $\int_\Gamma \langle x\rangle^2 |h|_2 \rho \leq C\|h\|_{W^{2,2}}^2$, as well as the elementary inequality $2ab\leq \epsilon a^2 + \epsilon^{-1} b^2$. 

Now for small enough $\epsilon_0$, we certainly have $\|U\|_{L^2}<1$ and $\|\phi_U\|_{W^{1,2}}<1$, so lower powers dominate. Using Proposition \ref{prop:ord1} then gives 

\begin{equation}
C^{-1} \delta_0 \|J'\|_{L^2}^2 \leq \| |\nabla^\perp\tau|^2_U\|_{L^1} + \epsilon \|U\|_{L^2}^2 + C(\epsilon,\kappa) \left(\|\phi_U\|_{L^2}^\frac{6}{3+\kappa}   + \|U_0\|_{L^2}^2 \right).
\end{equation}

Since \begin{equation}\|U\|_{L^2}^2 \leq \|U_0\|^2_{L^2} + \|h\|_{L^2}^2 + \|J'\|_{L^2}^2 \leq \|U_0\|_{L^2}^2  + C(\|\phi_U\|_{L^2}^2  +\|U\|_{L^2}^4) + \|J'\|_{L^2}^2, \end{equation}

if we choose $\epsilon < \frac{C}{\delta_0}$ then the $\|U\|_{L^2}^2$ term may be absorbed into the left hand side; thus
 
  \begin{equation}
 \|U\|_{L^2}^2 \leq C(\kappa) (\| |\nabla^\perp\tau|^2_U\|_{L^1}  +\|\phi_U\|_{L^2}^\frac{6}{3+\kappa} +  \|U_0\|_{L^2}^2). 
 \end{equation}
\end{proof}

\begin{remark}
Proposition \ref{prop:entire-est-tau-phi} is analogous to \cite[Proposition 4.47]{CM19} and to \cite[Proposition 2.1]{CM15}. Indeed, following the cutoff and rotation method we use in the proof of Theorem \ref{thm:lojasiewicz-tau} below, or in \cite[Theorem 7.1]{Z20}, yields a corresponding estimate by $|\nabla^\perp\tau|$ and $\phi$ for graphs over a (sufficiently large) subdomain. 
\end{remark}

\subsection{Estimates by $\phi$}

We now observe as in \cite{CM19} that the shrinker quantity $\phi$ controls the auxiliary quantity $\nabla^\perp\tau$, to a degree consistent with a second order obstruction. 

First, we need a PDE estimate proven by Colding-Minicozzi \cite{CM19}: 

\begin{proposition}[\cite{CM19}]
\label{prop:dtauP}
Let $\Gamma = \mathring{\Gamma}\times \mathbb{R}^{n-k}$, where $\mathring{\Gamma}$ is a round shrinking sphere or an Abresch-Langer curve. There exist $\delta>0$ and $C=C(\|U\|_{C^3})$ so that whenever $\|U\|_{C^2}<\delta$, we have
\begin{equation}\||\nabla^\perp\tau|^2_U\|_{L^1} \leq C( \|P_U\|_{L^1}  + \|\phi_U\|_{W^{2,1}} + \|\phi_U\|_{W^{1,2}}^2).\end{equation}
\end{proposition}
\begin{proof}
Since $U$ is compactly supported, we may apply \cite[Theorem 2.2]{CM19} to $\Gamma_U$ (with $\psi=1$). The desired estimate follows immediately after using the identity $\nabla^\perp_i \mathbf{H} = -\frac{1}{2}A(x^T, X_i)-\nabla^\perp_i \phi$ (cf. \cite[Proof of Theorem 7.4]{CM19}) and that $A$, $|\mathbf{H}|^{-1}$ and $|\nabla \mathbf{H}|$ are uniformly bounded. 
\end{proof}

Second, we have the following Taylor expansion estimate for $P$:

\begin{proposition}
\label{prop:Pest}
Let $\Gamma = \mathring{\Gamma}\times \mathbb{R}^{n-k}$, where $\mathring{\Gamma}$ is a round shrinking sphere or an Abresch-Langer curve. There exist $C,C_\kappa$ so that for any $\kappa\in(0,1]$, we have
\[ \|P_U\|_{L^1} \leq C_\kappa (\|U\|^3_{L^2} + \|\phi_U\|^\frac{6}{3+\kappa}_{L^2}) + C\|U\|_{L^2} \|\phi_U\|_{L^2} + C\|\phi_U\|^2_{W^{1,2}}.\]
\end{proposition}
\begin{proof}
The case of round cylinders was proven in \cite[Proposition 6.1]{CM19}, and the proof for Abresch-Langer cylinders proceeds exactly the same way. For the readers' convenience, we emphasise the corresponding ingredients:

Set $P(s) = \mathcal{P}(p,sU, s\nabla U,s\nabla^2 U)$; then $P(0)=0$ and Lemma \ref{lem:dPAL} and Proposition \ref{prop:d2PAL} imply that $P'(0)=0$ and \[|P''(0)| \leq C\langle x\rangle^2 |h|_2 (|J|_2 + |h|_2)\] respectively. Lemma \ref{lem:C3tauP} then gives the Taylor expansion estimate 
\[|P_U| \leq C\langle x\rangle^2 |h|_2 (|J|_2 + |h|_2) + C\langle x\rangle^6 |U|_2^3.\]

The proof now follows as in \cite[Proposition 6.1]{CM19}, using Corollary \ref{cor:jacobi-est} to estimate the $J$ terms and Proposition \ref{prop:ord1} for the $h, U$ terms. 
\end{proof}

We may now prove the main estimate for entire graphs:

\begin{theorem}
\label{thm:entire-est-tau}
Let $\Gamma = \mathring{\Gamma}\times \mathbb{R}^{n-k}$, where $\mathring{\Gamma}$ is a round shrinking sphere or an Abresch-Langer curve. There exists $\epsilon_0>0$ such that if $\kappa \in (0,1]$ and $U$ is a compactly supported normal vector field on $\Gamma$ with $\|U\|_{C^2}\leq \epsilon_0$ and $\|U\|_{C^3}\leq M$, then 
  \begin{equation}
\|U\|_{L^2}^2 \leq C(\kappa,M) (\|U_0\|_{L^2}^2 + \|\phi_U\|_{W^{2,1}} + \|\phi_U\|^2_{W^{1,2}}+ \|\phi_U\|_{L^2}^\frac{6}{3+\kappa}),
 \end{equation}
 where $U_0 = \pi_{\mathcal{K}_0}(U)$. 
 \end{theorem}
 \begin{proof}
Combining Propositions \ref{prop:entire-est-tau-phi}, \ref{prop:dtauP} and \ref{prop:Pest} and noting that lower powers dominate, we find that
  \begin{equation}
 \|U\|_{L^2}^2 \leq C(\kappa) (\|U\|^3_{L^2}+ \|U\|_{L^2} \|\phi_U\|_{L^2} + \|\phi_U\|^2_{W^{1,2}} + \|\phi_U\|_{W^{2,1}}  +\|\phi_U\|_{L^2}^\frac{6}{3+\kappa}  +  \|U_0\|_{L^2}^2). 
 \end{equation}
  Using the elementary inequality $2ab\leq \epsilon a^2 + \epsilon^{-1} b^2$ on the second term on the right and absorbing the resulting $\|U\|_{L^2}$ terms into the left hand side gives the result. 
 \end{proof}

\section{{\L}ojasiewicz inequalities via $\tau$}
\label{sec:main-est}

In this section, we conclude the {\L}ojasiewicz inequalities from the main estimate Theorem \ref{thm:entire-est-tau}, using the rotation and cutoff procedure in \cite[Section 7]{Z20}. For convenience set $\delta_R := R^n \e^{-R^2/4}$. 

\begin{theorem}
\label{thm:lojasiewicz-tau}
 Let $\mathring{\Gamma}$ be a round shrinking sphere or an Abresch-Langer curve. 
 
There exists $\epsilon_2>0$ so that for any $\epsilon_1$, $\lambda_0$, $C_j$ there exist $R_0, l_0$  such that if $l\geq l_0$, $\Sigma^n\subset \mathbb{R}^N$ has $\lambda(\Sigma)\leq \lambda_0$ and:
\begin{enumerate}
\item For some $R>R_0$, we have that $B_R\cap \Sigma$ is the graph of a normal field $U$ over some cylinder in $\mathcal{C}_n(\mathring{\Gamma})$ with $\|U\|_{C^{2}(B_R)} \leq \epsilon_2$ and $\|U\|_{L^2(B_R)} \leq \epsilon_2/R$; 
\item $|\nabla^j A| \leq C_j$ on $B_R\cap \Sigma$ for all $j \leq l$;
\end{enumerate}
then there is a cylinder $\Gamma \in\mathcal{C}_n(\mathring{\Gamma})$ and a compactly supported normal vector field $V$ over $\Gamma$ with $\|V\|_{C^{2,\alpha}} \leq \epsilon_1$, such that $\Sigma \cap B_{R-6}$ is contained in the graph of $V$, and 

\[
 \|V\|_{L^2}^2 \leq C\left(  \|U\|_{L^2}^{4a_l} +\|\phi\|_{L^1(B_R)}^{a_l} + \|\phi\|_{L^2(B_R)}^{2a_l}+ \delta_{R-5}^{a_l} \right),
\]

where $C=C(n,l,C_l,\lambda_0, \epsilon_1)$ and $a_l \nearrow 1$ as $l\to \infty$. 
\end{theorem}
\begin{proof}
Let $a_l := a_{l,2,n}$ be the exponent from interpolation (see Appendix \ref{sec:interpolation}). Following precisely the proof of \cite[Theorem 7.1]{Z20}, except using Theorem \ref{thm:entire-est-tau} in place of \cite[Theorem 5.8]{Z20} yields a vector field $V$, supported on $\Gamma \cap B_{R-5}$ and such that $\Sigma\cap B_{R-6}$ is contained in the graph of $V$, satisfying the estimate
\begin{equation}
\label{eq:loj-tau1}
\|V\|_{L^2}^2 \leq C(\|\pi_{\mathcal{K}_0}(V)\|_{L^2}^2 + \|\phi_V\|_{W^{2,1}} + \|\phi_V\|_{W^{1,2}}^2 + \|\phi_V\|_{L^2}^\frac{6}{3+\kappa}), 
\end{equation}
where the rotation part satisfies \begin{equation}
\label{eq:projV}
\|\pi_{\mathcal{K}_0}(V)\|_{L^2} \leq C(\|\phi\|_{L^2(B_R)}^{2a_l} + \|U\|_{L^2}^{2a_l} + \delta_{R-2}^{a_l} + \delta_{R-5}).
\end{equation}
Moreover, for any $s\in[1,2]$ we have the cutoff estimate \begin{equation}\label{eq:cutoff}\|\phi_{V}\|_{L^s}^s \leq \|\phi\|_{L^s(B_R)}^s +C(s)\delta_{R-5},\end{equation}

Using interpolation on the $\phi_V$ terms in (\ref{eq:loj-tau1}) and then the cutoff estimate, we have 
\begin{equation}
\label{eq:main-estimate}
\begin{split}
\|V\|_{L^2}^2  \leq {}& C( \|\phi\|_{L^2(B_R)}^{4a_l} + \delta_{R-2}^{2a_l} + \|U\|_{L^2}^{4a_l} + \delta_{R-5}^2) \\&+ C(  \|\phi\|_{L^1(B_R)}^{a_l} + \delta_{R-5}^{a_l} + \|\phi\|_{L^2(B_R)}^{2a_l}+\delta_{R-5}^{2a_l}+  \|\phi\|_{L^2(B_R)}^\frac{6}{3+\kappa} + \delta_{R-5}^\frac{6}{3+\kappa}).
\end{split}
\end{equation}

We take $\kappa$ so that $\frac{3}{3+\kappa}>a_l$. Then since lower powers dominate, we have

\begin{equation}
 \|V\|_{L^2}^2 \leq C\left(  \|U\|_{L^2}^{4a_l} +\|\phi\|_{L^1(B_R)}^{a_l} + \|\phi\|_{L^2(B_R)}^{2a_l}+ \delta_{R-5}^{a_l} \right).
\end{equation}
\end{proof}

\begin{proof}[Proof of Theorem \ref{thm:lojasiewicz-tau-intro}]
Apply Theorem \ref{thm:lojasiewicz-tau} and note that $\|U\|_{L^2}^{4a_l}$ is dominated by the exponential error term for large enough $l$. 
\end{proof}

\begin{proof}[Proof of Theorem \ref{thm:lojasiewicz-grad}]
Let $V,\Gamma$ be as given by Theorem \ref{thm:lojasiewicz-tau-intro}. Exactly as in the proof of \cite[Theorem 1.4]{Z20}, we have
\begin{equation}
|F(\Gamma_V)-F(\Gamma)|  \leq C( \|\phi_V\|_{L^2}^\frac{3}{2} + \|V\|_{L^2}^3),
\end{equation}
and $|F(\Sigma)-F(\Gamma_V)| \leq C\delta_{R-6}$. Consider $l$ large enough so that $a_l > \frac{2}{3}$. The conclusion of Theorem \ref{thm:lojasiewicz-tau-intro}, together with H\"{o}lder's inequality imply that \[\|V\|_{L^2}^3 \leq C(\|\phi\|_{L^2}^\frac{3a_l}{2} + \delta_{R-5}^\frac{3a_l}{2}).\] Using the cutoff estimate (\ref{eq:cutoff}) and collecting dominant terms, we conclude that
\begin{equation}
|F(\Sigma)-F(\Gamma)| \leq C(\|\phi\|_{L^2}^\frac{3a_l}{2} + \delta_{R-6}). 
\end{equation}

\end{proof}

\begin{remark}
\label{rmk:applications}
If in addition to the hypotheses of Theorem \ref{thm:lojasiewicz-tau-intro} one also has $\|\phi\|_{L^2} \leq \e^{-R^2/4}$, then arguing as in the proof of \cite[Theorem 7.2]{Z20} we obtain
\begin{equation}\|V\|_{L^2}^2 \leq C\left(  \e^{-a_l \frac{R^2}{4}}+ (R-5)^{a_l n} \e^{-a_l \frac{(R-5)^2}{4}}\right).
\end{equation}
Since $l$ may be chosen so that $a_l$ is arbitrarily close to 1, this gives an alternative proof of \cite[Theorem 7.2]{Z20}, and hence of \cite[Theorems 1.1 and 1.2]{Z20}, for the special cases where $\Gamma$ is a round cylinder or an Abresch-Langer cylinder. 
\end{remark}
 
\appendix

\section{Interpolation}
\label{sec:interpolation}

Here we recall some interpolation inequalities; see also \cite[Appendix A]{Z20} and \cite[Appendix B]{CM15}. In this appendix, $L^p$ refers to unweighted space, with $L^p_\rho$ the $\rho$-weighted space. 
  
  \begin{lemma}
There exists $C=C(m,j,n)$ so that if $u$ is a $C^m$ function on $B^n_{r}$, then for $j\leq m$, setting $a_{m,j,n} = \frac{m-j}{m+n}$ we have
\[r^j \|\nabla^j u\|_{L^\infty(B_r)} \leq C\left( r^{-n} \|u\|_{L^1(B_{r})} + r^j \|u\|_{L^1(B_{r})}^{a_{m,j,n}} \|\nabla^m u \|^{1-a_{m,j,n}}_{L^\infty(B_{r})}\right).\]
\end{lemma}

The lemma above also holds for tensor quantities on a manifold with uniformly bounded geometry. It follows that on a generalised cylinder $\Gamma = \mathring{\Gamma}^k\times \mathbb{R}^{n-k}$, for large enough $R$
\begin{equation}\|u\|_{W^{j,p}_\rho(B_{R-1}) }\leq C(m,j,n,p,M_0,M_m) \|u\|_{L^p_\rho(B_R)}^{a_{m,j,n}}.\end{equation}

\bigskip

\bibliographystyle{plain}
\bibliography{taylor}

\end{document}